\documentclass[12pt]{article}
\usepackage{amsthm}
\usepackage{amssymb}
\usepackage{amsmath}

\newtheorem{thm}{Theorem}
\newtheorem{prp}{Proposition}
\newtheorem{defi}{Definition}

\newtheorem{rem}{Remark}

\newtheorem{lem}{Lemma}

\newcommand{\idem}{\operatorname{idem}}

\begin{document}
\title{The Stokes phenomenon for the $q$-difference equation satisfied by the basic hypergeometric series ${}_3\varphi_1(a_1,a_2,a_3;b_1;q,x)$}
\author{Takeshi MORITA\thanks{Graduate School of Information Science and Technology, Osaka University, 
1-1  Machikaneyama-machi, Toyonaka, 560-0043, Japan.} }
\date{}
\maketitle
\begin{abstract}
We show the connection formula for the basic hypergeometric series ${}_3\varphi_1(a_1,a_2,a_3;b_1;q,x)$ between around the origin and infinity by the using of the $q$-Borel-Laplace transformations. We also show the limit $q\to 1-0$ of the new connection formula.
\end{abstract}

\section{Introduction}
In this paper, we show the connection formula for the \textit{divergent} basic hypergeometric series
\begin{equation}
{}_3\varphi_1(a_1,a_2,a_3;b_1;q,x)=\sum_{n\ge 0}\frac{(a_1,a_2,a_3;q)_n}{(b_1;q)_n(q;q)_n}\left\{(-1)^nq^{\frac{n(n-1)}{2}}\right\}^{-1}x^n\label{their}
\end{equation}
between around the origin and around infinity by the using of the $q$-Borel-Laplace resummation methods. Here, the function $(a;q)_n$ is the $q$-shifted factorial (see section 2 and \cite{GR} for more details of the $q$-shifted factorials and the basic hypergeometric series ${}_r\varphi_s(a_1,\dots ,a_r;b_1,\dots ,b_s;q,x)$):
\[(a;q)_n:=
\begin{cases}
1, &n=0, \\
(1-a)(1-aq)\dots (1-aq^{n-1}), &n\ge 1.
\end{cases}
\]

The series \eqref{their} satisfy the third order linear $q$-difference equation
\begin{align}
\left(a_1a_2a_3x-\frac{b_1}{q^2}\right)u(q^3x)&-\left\{(a_1a_2+a_2a_3+a_3a_1)x-\left(\frac{b_1}{q^2}+\frac{1}{q}\right)\right\}u(q^2x)\notag\\
&+\left\{(a_1+a_2+a_3)x-\frac{1}{q}\right\}u(qx)-xu(x)=0.\label{third} 
\end{align} 
Equation \eqref{third} also has a fundamental system of solutions around infinity:
\begin{align}
v_1(x)&:=x^{-\alpha_1}{}_3\varphi_2\left(a_1,\frac{a_1q}{b_1},0;\frac{a_1q}{a_2},\frac{a_1q}{a_3};q,\frac{qb_1}{a_1a_2a_3x}\right)\label{si1}\\
v_2(x)&:=x^{-\alpha_2}{}_3\varphi_2\left(a_2,\frac{a_2q}{b_1},0;\frac{a_2q}{a_1},\frac{a_2q}{a_3};q,\frac{qb_1}{a_1a_2a_3x}\right)\label{si2}\\
v_3(x)&:=x^{-\alpha_3}{}_3\varphi_2\left(a_3,\frac{a_3q}{b_1},0;\frac{a_3q}{a_2},\frac{a_3q}{a_1};q,\frac{qb_1}{a_1a_2a_3x}\right)\label{si3}
\end{align}
where $a_j=q^{\alpha_j}$, $j=1,2$ and $3$. In section 3, we show the connection formula between \eqref{si1}, 
\eqref{si2} \eqref{si3} and \eqref{their}.

\bigskip

We review the connection problems on the linear $q$-difference equations. Connection problems on the linear $q$-difference equations with regular singular points were studied by G.~D.~Birkhoff \cite{Birkhoff}. Connection formulae for the second order linear $q$-difference equations are given by the matrix form
\[\begin{pmatrix}
u_1(x)\\
u_2(x)
\end{pmatrix}
=
\begin{pmatrix}
C_{11}(x)&C_{12}(x)\\
C_{21}(x)&C_{22}(x)
\end{pmatrix}
\begin{pmatrix}
v_1(x)\\
v_2(x)
\end{pmatrix}.\]
The pair ($u_1(x), u_2(x)$) is a fundamental system of solutions around the origin and the pair $(v_1(x), v_2(x))$ is a fundamental system of solutions around infinity. The connection coefficients $C_{jk}(x)$ $(1\le j,k\le 2)$ are given by $q$-periodic and unique valued functions
\[\sigma_qC_{jk}(x)=C_{jk}(x),\quad C_{jk}(e^{2\pi i}x)=C_{jk}(x),\]
namely, the \textit{elliptic functions}.

The first example of the connection formula was given by G.~N.~Watson \cite{W} in 1910. Watson gave the connection formula for Heine's basic hypergeometric series 
\[{}_2\varphi_1(a,b;c;q,x):=\sum_{n\ge 0}\frac{(a,b;q)_n}{(c;q)_n(q;q)_n}x^n\]
around the origin and around the infinity \cite[page 117]{GR}. 
Heine's ${}_2\varphi_1(a,b;c;q,x)$ satisfies the $q$-difference equation 
\begin{equation}
\left[(c-abqx)\sigma_q^2-\left\{(c+q)-(a+b)qx\right\}\sigma_q+q(1-x)\right]u(x)=0.\label{heineeq}
\end{equation}
The equation \eqref{heineeq} also has a fundamental system of solutions around the infinity:
\[y_{\infty}^{(a,b)}(x)=x^{-\alpha}{}_2\varphi_1\left(a,\frac{aq}{c};\frac{aq}{b};q,\frac{cq}{abx}\right)\]
and
\[y_{\infty}^{(b,a)}(x)=x^{-\beta}{}_2\varphi_1\left(b,\frac{bq}{c};\frac{bq}{a};q,\frac{cq}{abx}\right),\]
provided that $a=q^\alpha$ and $b=q^{\beta}$. 
Watson's connection formula for ${}_2\varphi_1(a,b;c;q,x)$ is given by
\begin{align}\label{wato}
{}_2 \varphi_1\left(a,b;c;q;x \right)&= 
\frac{(b,c/a;q)_\infty \theta (-ax)_\infty }{(c, b/a;q)_\infty \theta (-x)_\infty }\frac{\theta (x)}{\theta (ax)} 
y_{\infty}^{(a,b)}(x) \nonumber \\
&+\frac{(a,c/b;q)_\infty \theta (-bx)_\infty }{(c, a/b;q)_\infty \theta (-x)_\infty } \frac{\theta (x)}{\theta (bx)} 
y_{\infty}^{(b,a)}(x).\notag
\end{align}
Here, the notation $\theta (x)$ is the theta function of Jacobi(see section two for more details). We remark that the connection coefficients are given by the $q$-elliptic functions.

\bigskip
But connection formulae for $q$-difference equations with irregular singular points had not known for a long time. We remark that A.~Duval and C.~Mitschi gave connection matrices for degenerated \textit{differential} equations \cite{Du-Mi}. 
The irregularity of $q$-difference equations are studied by the using of the Newton polygons by J.-P.~Ramis, J.~Sauloy and C.~Zhang \cite{RSZ}. C.~Zhang gave connection formulae for some confluent type basic hypergeometric series \cite{Z0,Z1,Z2} where he uses the $q$-Borel-Laplace transformations. In \cite{M0,M1}, the author gave the connection formula for the Hahn-Exton $q$-Bessel function and the $q$-confluent type function by the $q$-Borel-Laplace transformations. These resummation methods are powerful tools for connection problems on linear $q$-difference equations with irregular singular points.

\bigskip
\noindent
\begin{defi}
We assume that $f(x)$ is a formal power series $f(x)=\sum_{n\in\mathbb{Z}}a_nx^n$, $a_0=1$.
\begin{enumerate}
\item The $q$-Borel transformation is
\[\left(\mathcal{B}_q^+f\right)(\xi ):=\sum_{n\in\mathbb{Z}}a_nq^{\frac{n(n-1)}{2}}\xi^n\left(=:\psi (\xi )\right).\]
\item For any analytic function $\psi (\xi )$ around $\xi =0$, the $q$-Laplace transformation is
\[\left(\mathcal{L}_{q, \lambda}^+\psi\right)(x):=
\frac{1}{1-q}\int_0^{\lambda\infty}\frac{\varphi (\xi )}{\theta_q\left(\frac{\xi}{x}\right)}\frac{d_q\xi}{\xi}=\sum_{n\in\mathbb{Z}}\frac{\varphi (\lambda q^n)}{\theta_q\left(\frac{\lambda q^n}{x}\right)}.\]
Here, this transformation is given by Jackson's $q$-integral \cite[page 23]{GR}. 
\end{enumerate}
\end{defi}
The definition is a special case of one of the $q$-Laplace transformations in \cite{ZandD, Z0}. The $q$-Borel transformation is the formal inverse of the $q$-Laplace transformation as follows:
\begin{lem}[Zhang, \cite{Z0}]
For any entire function $f(x)$, we have
\[\mathcal{L}_{q,\lambda}^+\circ\mathcal{B}_q^+f=f.\]
\end{lem}
Thanks to these methods, some connection formulae for the second order $q$-difference equations were found. However, the connection formulae for more higher order linear $q$-difference equations have not known. In this paper, especially we apply the $q$-Borel-Laplace transformations to the divergent series \eqref{their} to study the connection problem on the third order $q$-difference equation. In the section 3, we show the following theorem:

\bigskip
\noindent
\textbf{Theorem.}\textit{ For any $x\in\mathbb{C}^*\setminus [-\lambda ;q]$, we have} 
\begin{align*}
&{}_3f_1(a_1,a_2,a_3;b_1;q;\lambda ,x):=
\left(\mathcal{L}_{q,\lambda}^+\circ\mathcal{B}_q^+ {}_3\varphi_1(a_1,a_2,a_3;b_1;q,x)\right)(x)\\
&=\frac{(a_2,a_3,b_1/a_1;q)_\infty}{(b_1,a_2/a_1,a_3/a_1;q)_\infty}\frac{\theta (a_1\lambda )}{\theta (\lambda )}\frac{\theta (a_1qx/\lambda )}{\theta (qx/\lambda )}\frac{\theta (x)}{\theta (a_1x)}v_1(x)\\
&+\frac{(a_1,a_3,b_1/a_2;q)_\infty}{(b_1,a_1/a_2,a_3/a_2;q)_\infty}\frac{\theta (a_2\lambda )}{\theta (\lambda )}\frac{\theta (a_2qx/\lambda )}{\theta (qx/\lambda )}\frac{\theta (x)}{\theta (a_2x)}v_2(x)\\
&+\frac{(a_2,a_1,b_1/a_3;q)_\infty}{(b_1,a_2/a_3,a_1/a_3;q)_\infty}\frac{\theta (a_3\lambda )}{\theta (\lambda )}\frac{\theta (a_3qx/\lambda )}{\theta (qx/\lambda )}\frac{\theta (x)}{\theta (a_3x)}v_3(x).
\end{align*}
\textit{Here, $\left(\mathcal{L}_{q,\lambda}^+\circ\mathcal{B}_q^+ {}_3\varphi_1(a_1,a_2,a_3;b_1;q,x)\right)(x)$ is the $q$-Borel-Laplace transform of the divergent series ${}_3\varphi_1(a_1,a_2,a_3;b_1;q,x)$.}

We remark that the connection coefficients(with the new parameter $\lambda$) are given by the $q$-elliptic functions.    These coefficients are also  the new example of the Stokes phenomenon \cite{ZandD} for the $q$-difference equation \eqref{third}.
 
\medskip
In the last section, we also give the limit $q\to 1-0$ of the new connection formula.

\section{Basic notations}
In this section, we review our notations. The $q$-shifted operator $\sigma_q$ is given by $\sigma_qf(x)=f(qx)$. For any fixed $\lambda\in\mathbb{C}^*\setminus q^{\mathbb{Z}}$, the set $[\lambda ;q]$-spiral is $[\lambda ;q]:=\lambda q^{\mathbb{Z}}=\{\lambda q^k;k\in\mathbb{Z}\}$. 
The function $(a;q)_n$ is the $q$-shifted factorial such that
\[(a;q)_n:=
\begin{cases}
1, &n=0, \\
(1-a)(1-aq)\dots (1-aq^{n-1}), &n\ge 1.
\end{cases}
\]
moreover, $(a;q)_\infty :=\lim_{n\to \infty}(a;q)_n$ and 
\[(a_1,a_2,\dots ,a_m;q)_\infty:=(a_1;q)_\infty (a_2;q)_\infty \dots (a_m;q)_\infty.\]

\noindent
The basic hypergeometric series with the base $q$ \cite[page 4]{GR} is
\begin{align*}
{}_r\varphi_s(a_1,\dots ,a_r&;b_1,\dots ,b_s;q,x)\\
&:=\sum_{n\ge 0}\frac{(a_1,\dots ,a_r;q)_n}{(b_1,\dots ,b_s;q)_n(q;q)_n}\left\{(-1)^nq^{\frac{n(n-1)}{2}}\right\}^{1+s-r}x^n.
\end{align*}
The radius of convergence is $\infty , 1$ or $0$ according to whether $r-s<1, r-s=1$ or $r-s>1$. 

\noindent
The theta function of Jacobi is important in connection problems on linear $q$-difference equations. The theta function with the base $q$ is
\[\theta_q(x):=\sum_{n\in\mathbb{Z}}q^{\frac{n(n-1)}{2}}x^n,\qquad \forall x\in\mathbb{C}^*.\]
The theta function has the triple product identity
\begin{equation}
\theta_q(x)=\left(q,-x,-\frac{q}{x};q\right)_\infty .
\label{triple}
\end{equation}
The theta function satisfies the $q$-difference equation
$\theta_q(q^kx)=q^{-\frac{n(n-1)}{2}}x^{-k}\theta_q(x)$, $\forall k\in\mathbb{Z}$. The theta function also has the inversion formula $\theta_q\left(1/x\right)=\theta_q(x)/x$. We remark that $\theta (\lambda q^k/x)=0$ if and only if $x\in [-\lambda ;q]$. The function $\theta (x)/\theta (q^\alpha x)$, $\forall\alpha\not\in \mathbb{Z}$ satisfies a $q$-difference equation
\[u(qx)=q^{\alpha}u(x),\]
which is also satisfied by the function $u(x)=x^{\alpha}$. 

\section{The connection formula}
In this section, we give the new connection formula for the basic hypergeometric series ${}_3\varphi_1(a_1,a_2,a_3;b_1;q,x)$. In section \ref{sec3.1}, we review the connection formula of non-degenerated series ${}_3\varphi_2(a_1,a_2,a_3;b_1,b_2;q,x)$. 

\subsection{The non-degenerated case}\label{sec3.1}
The non-degenerated convergent series 
\begin{equation}
{}_3\varphi_2(a_1,a_2,a_3;b_1,b_2;q,x):=\sum_{n\ge 0}\frac{(a_1,a_2,a_3;q)_n}{(b_1,b_2;q)_n(q;q)_n}x^n
\label{solo1}
\end{equation}
satisfies the third order $q$-difference equation
\begin{align}
&\left[\left(a_1a_2a_3x-\frac{b_1b_2}{q^2}\right)\sigma_q^3-\left\{(a_1a_2+a_2a_3+a_3a_1)x-\left(\frac{b_1b_2}{q^2}+\frac{b_2}{q}+\frac{b_1}{q}\right)\right\}\sigma_q^2\right.\label{ND3}\\
&\left. \left\{(a_1+a_2+a_3)x-\left(\frac{b_1}{q}+\frac{b_2}{q}+1\right)\right\}\sigma_q-(x-1)\right]u(x)=0.
\notag
\end{align}
Equation \eqref{ND3} also has a fundamental system of solutions around infinity:
\begin{align}
\tilde{v}_1(x)&=\frac{\theta (a_1x)}{\theta (x)}{}_3\varphi_2\left(a_1,\frac{a_1q}{b_1},\frac{a_1q}{b_2};\frac{a_1q}{a_2},\frac{a_1q}{a_3};q,\frac{qb_1b_2}{a_1a_2a_3x}\right),\label{soli1}\\
\tilde{v}_2(x)&=\frac{\theta (a_2x)}{\theta (x)}{}_3\varphi_2\left(a_2,\frac{a_2q}{b_1},\frac{a_2q}{b_2};\frac{a_2q}{a_1},\frac{a_2q}{a_3};q,\frac{qb_1b_2}{a_1a_2a_3x}\right),\label{soli2}\\
\tilde{v}_3(x)&=\frac{\theta (a_3x)}{\theta (x)}{}_3\varphi_2\left(a_3,\frac{a_3q}{b_1},\frac{a_3q}{b_2};\frac{a_3q}{a_2},\frac{a_3q}{a_1};q,\frac{qb_1b_2}{a_1a_2a_3x}\right).\label{soli3}
\end{align}
The connection formula between the solutions \eqref{soli1}, \eqref{soli2}, \eqref{soli3} and \eqref{solo1} can be found in \cite[page 121]{GR}. We remark that the following formula was essentially given by L.~J.~Slater.

\begin{thm}[Slater, \cite{Slater}]\label{Slater}For any $x\in\mathbb{C}^*$, we have
\begin{align*}{}_3\varphi_2(a_1,a_2,a_3;b_1,b_2;q,x)&=
\frac{(a_2,a_3,b_1/a_1,b_2/a_1;q)_\infty}{(b_1,b_2,a_2/a_1,a_3/a_1;q)_\infty}
\frac{\theta (-a_1x)}{\theta (-x)}\frac{\theta (x)}{\theta (a_1x)}\tilde{v}_1\\
&+\idem (a_1;a_2,a_3).
\end{align*}
Provided that the notation $\idem (a_1;a_2,a_3)$ after an expression stands for the sum expressions obtained from the preceding expression by interchanging $a_1$ with each $a_2$ and $a_3$.
\end{thm}
This Theorem can be considered as the higher order extension of Watson's formula. By Theorem \ref{Slater}, we obtain the following key Lemma.

\begin{lem}\label{ni}For any $x\in\mathbb{C}^*$, we have
\begin{align*}
&{}_3\varphi_2(a_1,a_2,a_3;b_1,0;q,x)\\
&=\frac{(a_2,a_3,b_1/a_1;q)_\infty}{(b_1,a_2/a_1,a_3/a_1;q)_\infty}
\frac{\theta (-a_1x)}{\theta (-x)}
{}_2\varphi_2\left(a_1,\frac{a_1q}{b_1};\frac{a_1q}{a_2},\frac{a_1q}{a_3};q,\frac{q^2b_1}{a_2a_3x}\right)\\
&+\idem (a_1;a_2,a_3).
\end{align*}
\end{lem}
\begin{proof} We tale the limit $b_2\to 0$ in Theorem \ref{Slater}, we obtain the conclusion.
\end{proof}
In the next section, we prove our new connection formula by Lemma \ref{ni} and the $q$-Borel-Laplace transformations.

\subsection{Proof of main Theorem}In this section, we prove the following Theorem.
\begin{thm}\label{main} For any $x\in\mathbb{C}^*\setminus [-\lambda ;q]$, we have
\begin{align*}
&{}_3f_1(a_1,a_2,a_3;b_1;q;\lambda ,x):=
\left(\mathcal{L}_{q,\lambda}^+\circ\mathcal{B}_q^+ {}_3\varphi_1(a_1,a_2,a_3;b_1;q,x)\right)(x)\\
&=\frac{(a_2,a_3,b_1/a_1;q)_\infty}{(b_1,a_2/a_1,a_3/a_1;q)_\infty}\frac{\theta (a_1\lambda )}{\theta (\lambda )}\frac{\theta (a_1qx/\lambda )}{\theta (qx/\lambda )}\frac{\theta (x)}{\theta (a_1x)}v_1(x)\\
&+\frac{(a_1,a_3,b_1/a_2;q)_\infty}{(b_1,a_1/a_2,a_3/a_2;q)_\infty}\frac{\theta (a_2\lambda )}{\theta (\lambda )}\frac{\theta (a_2qx/\lambda )}{\theta (qx/\lambda )}\frac{\theta (x)}{\theta (a_2x)}v_2(x)\\
&+\frac{(a_2,a_1,b_1/a_3;q)_\infty}{(b_1,a_2/a_3,a_1/a_3;q)_\infty}\frac{\theta (a_3\lambda )}{\theta (\lambda )}\frac{\theta (a_3qx/\lambda )}{\theta (qx/\lambda )}\frac{\theta (x)}{\theta (a_3x)}v_3(x).
\end{align*}
\end{thm}

\begin{proof}
We apply the $q$-Borel transformation to the series ${}_3\varphi_1(a_1,a_2,a_3;b_1;q,x)$.
\[\left(\mathcal{B}_q^+{}_3\varphi_1(a_1,a_2,a_3;b_1;q,x)\right)(\xi )={}_3\varphi_2(a_1,a_2,a_3;b_1,0,-\xi )=:\varphi (\xi ).\]
By Lemma \ref{ni}, we have another expression of the function $\varphi(\xi )$. We also apply the $q$-Laplace transformation $\mathcal{L}_{q,\lambda}$ to the function $\varphi (\xi )$, we obtain the conclusion. 
\end{proof}
\begin{rem}We remark that the fundamental system of solutions for equation \eqref{third} are given by 
\begin{align}
v_1(x)&:=\frac{\theta (a_1x)}{\theta (x)}{}_3\varphi_2\left(a_1,\frac{a_1q}{b_1},0;\frac{a_1q}{a_2},\frac{a_1q}{a_3};q,\frac{qb_1}{a_1a_2a_3x}\right),\\
v_2(x)&:=\frac{\theta (a_2x)}{\theta (x)}{}_3\varphi_2\left(a_2,\frac{a_2q}{b_1},0;\frac{a_2q}{a_1},\frac{a_2q}{a_3};q,\frac{qb_1}{a_1a_2a_3x}\right),\\
v_3(x)&:=\frac{\theta (a_3x)}{\theta (x)}{}_3\varphi_2\left(a_3,\frac{a_3q}{b_1},0;\frac{a_3q}{a_2},\frac{a_3q}{a_1};q,\frac{qb_1}{a_1a_2a_3x}\right)
\end{align}
in the Theorem \ref{main}.
\end{rem}
\begin{rem}By the $q$-difference equation of the theta function, we can check out that the connection coefficients (with the new parameter $\lambda$)
\begin{align*}
C_1(x)&:=\frac{(a_2,a_3,b_1/a_1;q)_\infty}{(b_1,a_2/a_1,a_3/a_1;q)_\infty}\frac{\theta (a_1\lambda )}{\theta (\lambda )}\frac{\theta (a_1qx/\lambda )}{\theta (qx/\lambda )}\frac{\theta (x)}{\theta (a_1x)},\\
C_2(x)&:=\frac{(a_1,a_3,b_1/a_2;q)_\infty}{(b_1,a_1/a_2,a_3/a_2;q)_\infty}\frac{\theta (a_2\lambda )}{\theta (\lambda )}\frac{\theta (a_2qx/\lambda )}{\theta (qx/\lambda )}\frac{\theta (x)}{\theta (a_2x)},\\
C_3(x)&:=\frac{(a_2,a_1,b_1/a_3;q)_\infty}{(b_1,a_2/a_3,a_1/a_3;q)_\infty}\frac{\theta (a_3\lambda )}{\theta (\lambda )}\frac{\theta (a_3qx/\lambda )}{\theta (qx/\lambda )}\frac{\theta (x)}{\theta (a_3x)}
\end{align*}
are the $q$-elliptic functions.
\end{rem}

\section{The limit $q\to 1-0$ of the connection formula}
The aim of this section is to give the limit $q\to 1-0$ of the new connection formula as follows:

\begin{thm}\label{limit}For any $x\in\mathbb{C}^*\setminus [-\lambda ;q]$, we have the following limit $q\to 1-0$ of the connection formula
\begin{align*}
&\lim_{q\to 1-0}{}_3f_1(q^{\alpha_1},q^{\alpha_2},q^{\alpha_3};q^{\beta_1};q;\lambda ,x)\\
&=\frac{\Gamma (\beta_1)\Gamma (\alpha_2-\alpha_1)\Gamma (\alpha_3-\alpha_1)}{\Gamma (\alpha_2)\Gamma (\alpha_3)\Gamma(\beta_1-\alpha_1)}
x^{-\alpha_1}
{}_2F_2\left(\alpha_1,\alpha_1+1-\beta_1;\alpha_1+1-\alpha_2,\alpha_1+1-\alpha_3;\frac{1}{x}\right)\\
&+\frac{\Gamma (\beta_1)\Gamma (\alpha_1-\alpha_2)\Gamma (\alpha_3-\alpha_2)}{\Gamma (\alpha_1)\Gamma (\alpha_3)\Gamma(\beta_1-\alpha_2)}
x^{-\alpha_2}
{}_2F_2\left(\alpha_2,\alpha_2+1-\beta_1;\alpha_2+1-\alpha_1,\alpha_2+1-\alpha_3;\frac{1}{x}\right)\\
&+\frac{\Gamma (\beta_1)\Gamma (\alpha_2-\alpha_3)\Gamma (\alpha_1-\alpha_3)}{\Gamma (\alpha_2)\Gamma (\alpha_1)\Gamma(\beta_1-\alpha_3)}
x^{-\alpha_3}
{}_2F_2\left(\alpha_3,\alpha_3+1-\beta_1;\alpha_3+1-\alpha_2,\alpha_3+1-\alpha_1;\frac{1}{x}\right),
\end{align*}
provided that $-\pi<\arg x<\pi$.
\end{thm}

The following proposition \cite{Z1} is important to consider the limit $q\to 1-0$ of our connection formula.
\begin{prp} For any $x\in\mathbb{C}^* (-\pi <\arg x<\pi )$, we have
\begin{equation}
\lim_{q\to 1-0}\frac{\theta (q^\beta x)}{\theta (q^\alpha x)}=x^{\alpha -\beta}\label{limt1}
\end{equation}
and 
\begin{equation}
\lim_{q\to 1-0}\frac{\theta \left(\dfrac{q^\alpha x}{(1-q)}\right)}{\theta \left(\dfrac{q^\beta x}{(1-q)}\right)}(1-q)^{\beta -\alpha}=x^{\beta -\alpha}.\label{limt2}
\end{equation}
\end{prp}
\noindent
We also review the $q$-gamma function. The $q$-gamma function $\Gamma_q(x)$ is
\[\Gamma_q(x):=\frac{(q;q)_\infty}{(q^x;q)_\infty}(1-q)^{1-x},\qquad 0<q<1.\]
The limit $q\to 1-0$ of $\Gamma_q(x)$ gives the gamma gunction \cite[page 20]{GR}
\begin{equation}
\lim_{q\to 1-0}\Gamma_q(x)=\Gamma (x).\label{limgamma}
\end{equation}

We give the proof of the Theorem \ref{limit}.
\begin{proof}
At first, we put $a_j:=q^{\alpha_j} $ ($j=1,2,3$), $b_1:=q^{\beta_1}$ and $x\mapsto x/(1-q)$. 
We remark that the limit $q\to 1-0$ of the left hand-side of Theorem \ref{limit} formally converges the hypergeometric series 
\[{}_3F_1(\alpha_1,\alpha_2,\alpha_3;\beta_1;x)=\sum_{n\ge 0}\frac{(\alpha_1,\alpha_2,\alpha_3)_n}{(\beta_1)_n n!}x^n.\]
We consider the right hand-side. The connection formula can be rewritten as follows:
\begin{align*}
&{}_3f_1(q^{\alpha_1},q^{\alpha_2},q^{\alpha_3};q^{\beta_1};q;\lambda ,x)\\
&=\frac{(q^{\alpha_2},q^{\alpha_3},q^{\beta_1-\alpha_1};q)_\infty}{(q^{\beta_1},q^{\alpha_2-\alpha_1},q^{\alpha_3-\alpha_1};q)_\infty}\frac{\theta (q^{\alpha_1}\lambda )}{\theta (\lambda )}\frac{\theta \left(\frac{q^{\alpha_1+1}x}{\lambda (1-q)}\right)}{\theta \left(\frac{qx}{\lambda (1-q)} \right)}\\
&\times{}_3\varphi_2\left(q^{\alpha_1},q^{\alpha_1+1-\beta_1},0;q^{\alpha_1+1-\alpha_2},q^{\alpha_1+1-\alpha_3};q,\frac{q^{1+\beta_1}(1-q)}{q^{\alpha_1+\alpha_2+\alpha_3}x}\right)\\
&+\idem (q^{\alpha_1}; q^{\alpha_2},q^{\alpha_3})\\
&=\frac{\Gamma_q(\beta_1)\Gamma_q(\alpha_2-\alpha_1)\Gamma_q(\alpha_3-\alpha_1)}{\Gamma_q(\alpha_2)\Gamma_q(\alpha_3)\Gamma_q(\beta_1-\alpha_1)}
\frac{\theta (q^{\alpha_1}\lambda )}{\theta (\lambda )}
\left\{\frac{\theta \left(\frac{q^{\alpha_1+1}x}{\lambda (1-q)}\right)}{\theta \left(\frac{qx}{\lambda (1-q)} \right)}(1-q)^{-\alpha_1}\right\}\\
&\times{}_3\varphi_2\left(q^{\alpha_1},q^{\alpha_1+1-\beta_1},0;q^{\alpha_1+1-\alpha_2},q^{\alpha_1+1-\alpha_3};q,\frac{q^{1+\beta_1}(1-q)}{q^{\alpha_1+\alpha_2+\alpha_3}x}\right)\\
&+\idem (q^{\alpha_1}; q^{\alpha_2},q^{\alpha_3}).
\end{align*}
By \eqref{limt1}, \eqref{limt2} and \eqref{limgamma}, we obtain the conclusion. 
\end{proof}

\section*{Acknowledgements}
The author would like to give heartful thanks to Professor Yousuke Ohyama who provided carefully considered feedback and many valuable comments. The author also would like to show his greatest appreciation to Professor Masahiko Ito who provided helpful comments and suggestions.

\end{document}